\documentclass[a4paper,12pt]{article}

\usepackage{amsmath}
\usepackage{amsthm}
\usepackage{amssymb}

\usepackage{latexsym}
\usepackage{array}
\usepackage{enumerate}

\numberwithin{equation}{section}

\newcommand{\CC}{\mathbb{C}}

\newcommand{\RR}{\mathbb{R}}
\newcommand{\QQ}{\mathbb{Q}}
\newcommand{\ZZ}{\mathbb{Z}}
\newcommand{\D}{\mathcal{D}}

\renewcommand{\d}{{\rm dim}}

\newcommand{\e}{\varepsilon}

\newcommand{\supp}{{\rm supp}}

\newcommand{\simto}{\overset{\sim}{\longrightarrow}}

\newtheorem{definition}{Definition}[section]
\newtheorem{theorem}[definition]{Theorem}
\newtheorem{proposition}[definition]{Proposition}
\newtheorem{lemma}[definition]{Lemma}
\newtheorem{corollary}[definition]{Corollary}

\title{Coefficients of the poles of local zeta 
functions and their applications to oscillating integrals 
}

\author{Toshihisa \textsc{Okada}\footnote{
 Ibaraki, Japan.} \and 
Kiyoshi \textsc{Takeuchi}\footnote{Institute of Mathematics, 
University  of 
Tsukuba, 1-1-1, Tennodai, Tsukuba, 
Ibaraki, 305-8571, Japan. 
e-mail: takemicro@nifty.com}}

\date{}

\begin{document}

\maketitle

\begin{abstract}
We introduce a new method which enables us to 
calculate the coefficients of the poles of local 
zeta functions very precisely and prove some 
explicit formulas. Some vanishing theorems 
for the candidate poles of local zeta 
functions will be also given. 
Moreover we apply our method 
to oscillating integrals and obtain an explicit 
formula for the coefficients of their asymptotic 
expansions. 
\end{abstract}

\section{Introduction}\label{sec:1}

The theory of local zeta functions is an important 
subject in many fields of mathematics, such as 
 generalized functions, number theory and prehomogeneous 
vector spaces. By the fundamental work \cite{V} 
due to Varchenko, we can find a subset 
$P$ of $\QQ_{<0}$ in which the poles of a local 
zeta function are contained (see Section \ref{sec:3} 
and \cite{A-G-Z-V} etc. for the detail). 
 After this pioneering paper \cite{V} 
many authors studied the poles of local zeta 
functions. However, to the best of our knowledge,  
there is almost no paper which calculated 
the coefficients of these poles explicitly 
in a general setting. Note that 
in a remarkable work \cite{I} Igusa could 
calculate these coefficients when the local 
zeta functions are associated to the relative 
invariants of some prehomogeneous vector 
spaces. He calculated the coefficients by 
using some group actions. 

 In this paper, we propose a 
new method which enables us to 
calculate the coefficients of the poles of local 
zeta functions as precisely as we want. 
The key idea is the use of 
 a meromorphic continuation 
of the distribution 
\begin{equation}
x_{1+}^{l_1 \lambda +m_1} x_{2+}^{l_2 \lambda +m_2} 
\cdots x_{n+}^{l_n \lambda +m_n} 
\in \D^{\prime}(\RR^n) 
\end{equation}
($l_1,l_2, \ldots , l_n \in \RR_{>0}$ and 
$m_1, m_2, \ldots , m_n \in \RR_+
=\RR_{\geq 0}$) with 
respect to the complex parameter $\lambda$, 
which is different from the usual one used in 
the study of local zeta functions (see 
\cite{A-G-Z-V} etc.). See Section \ref{sec:2} 
for the detail. This meromorphic 
continuation is useful and enables us to 
get very precise information on the poles 
of local zeta functions. In order to state 
our results, now let us recall the definition 
of local zeta functions. 
Let $f$ be a real-valued real analytic function 
defined on an open neighborhood $U$ 
of $0 \in \RR^n$ and $\varphi \in C_0^{\infty}(U)$ 
a real-valued test function 
$\varphi \in C_0^{\infty}(U)$ on $U$. Then 
the integral 
\begin{equation}
Z_f(\varphi)(\lambda) =\int_{\RR^n} 
 |f(x)|^{\lambda} \varphi (x)dx 
\end{equation}
converges locally uniformly 
on $\{ \lambda \in \CC \ | \ 
\Re \lambda >0\}$ and defines a holomorphic 
function there. Moreover 
 $Z_f(\varphi)(\lambda)$ can be extended 
to a meromorphic function on the whole 
complex plane $\CC$ (see 
\cite{A-G-Z-V} and \cite{V} etc.). 
This meromorphic function $Z_f(\varphi)$ 
of $\lambda$ is called the local zeta 
function associated to 
$f$ and $\varphi$. As is clear from the 
proofs of our results in this paper, 
by our method we can calculate 
the coefficients of the poles 
of $Z_f(\varphi)$ very precisely 
once a good resolution of 
the singularities of the hypersurface 
$\{ x \in U  \ | \  f(x)=0 \}$ is 
obtained. It would be possible 
to calculate them as precisely 
as we want by numerical 
computations. However, since the 
general formula (which is evident 
from our proofs) is involved, we 
restrict ourselves here to a generic 
case where the formula can be 
stated neatly. From now on, 
 assume that the hypersurface 
$\{ x \in U   \ | \ f(x)=0 \}$ has an isolated 
singular point at $0 \in U \subset \RR^n$ 
and $f$ is convenient and 
non-degenerate (in the sense 
of Definition \ref{ND}). 
Let $\Sigma_0$ be the 
dual fan of the Newton polygon 
$\Gamma_+(f)$ of $f$ and $\Sigma$ 
a smooth subdivision of $\Sigma_0$. Recall that 
 $\Sigma =\{ \sigma \}$ 
is a family of rational polyhedral convex 
cones in $\RR_+^n=\RR^n_{\geq 0}$ such that $\RR_+^n= 
\bigcup_{\sigma \in \Sigma}\sigma$. 
For a cone $\sigma$ in $\Sigma$ let 
 $\{  a^1(\sigma), a^2(\sigma), 
\ldots , a^{\d \sigma}(\sigma)  \} 
\subset 
\partial \sigma \cap (\ZZ^n 
\setminus \{ 0\})$ 
be the $1$-skelton of 
$\sigma$. For each 
$1 \leq i \leq \d \sigma$ set 
\begin{equation}
l(a^i(\sigma))=\min_{\alpha \in \Gamma_+(f)} 
\langle a^i(\sigma), \alpha \rangle \in \ZZ_+ 
\end{equation}
and 
\begin{equation}
 |a^i(\sigma)| =\sum_{j=1}^n a^i(\sigma)_j 
\in \ZZ_{>0}. 
\end{equation}
For $0 \leq k \leq n$ let $\Sigma^{(k)} 
\subset \Sigma$ be the subset of 
$\Sigma$ consisting of $k$-dimensional 
cones. 
\begin{definition}
(see \cite{A-G-Z-V} and \cite{V} etc.) 
Let $P \subset \QQ_{<0}$ be the union of 
the following subsets of $\QQ_{<0}$: 
\begin{equation}
\{ -1,-2, -3, \cdots \cdots \}, 
\end{equation}
\begin{equation}
\left\{ -\frac{|a^1(\sigma)|}{l(a^1(\sigma))}, 
-\frac{|a^1(\sigma)|+1}{l(a^1(\sigma))}, 
\cdots \cdots \right\} 
\qquad 
(\sigma \in \Sigma^{(1)} 
 \ \text{such that} \ l(a^1(\sigma))>0). 
\end{equation}
\end{definition} 
By the fundamental 
results of \cite{V}, the poles of 
 $Z_f(\varphi)$ are contained 
in $P$. We call an element of $P$ a 
candidate pole of $Z_f(\varphi)$. Let us order 
the candidate poles of $Z_f(\varphi)$ as 
\begin{equation}
P= 
\{ -\lambda_1 > -\lambda_2 > -\lambda_3 > 
 \cdots \cdots \} \qquad (\lambda_j \in \QQ_{>0}). 
\end{equation}

\begin{definition}
\par \noindent (i) 
Let $\sigma \in \Sigma$. For $1 \leq i \leq \d \sigma$ 
such that $l(a^i(\sigma))>0$ ($\Longleftrightarrow 
a^i(\sigma) \in \RR_{>0}^n$) we set 
\begin{equation}
K^i(\sigma)=\left\{ 
\frac{|a^i(\sigma)|}{l(a^i(\sigma))}, 
\frac{|a^i(\sigma)|+1}{l(a^i(\sigma))}, 
\cdots \cdots \right\} \subset \QQ_{>0}. 
\end{equation}
\par \noindent (ii) 
For a candidate pole $-\lambda_j \in P$ 
of $Z_f(\varphi)$ and $0 \leq k \leq n$ we 
define a subset $\Sigma^{(k)}_j$ of $\Sigma^{(k)}$ 
by 
\begin{equation}
\Sigma^{(k)}_j=\{ \sigma \in \Sigma^{(k)} \ | \ 
l(a^i(\sigma))>0 \quad \text{and} \quad 
\lambda_j \in K^i(\sigma) 
\quad \text{for} \quad 1 \leq i \leq k \}. 
\end{equation}
\par \noindent (iii) 
For $\sigma \in \Sigma^{(k)}_j$ and $1 \leq i \leq k$ 
we define a non-negative integer $\nu(\sigma)_i 
\in \ZZ_+$ by 
\begin{equation}
\lambda_j=\frac{ |a^i(\sigma)| + \nu(\sigma)_i}{ 
l(a^i(\sigma))}. 
\end{equation}
\end{definition}

For the sake of simplicity, 
in this introduction we 
assume that $-\lambda_j \in P$ is not 
an integer. Then 
it is well-known that 
the order of the pole of $Z(\varphi)$ at 
$\lambda =-\lambda_j$ is less than or equal to 
\begin{equation}
k_j:= \max \{ 0 \leq k \leq n \ | \ 
\Sigma_j^{(k)} \not= \emptyset \} \in \ZZ_+. 
\end{equation}
Let 
\begin{equation}
\frac{a_{j,k_j}(\varphi) }{(\lambda +\lambda_j)^{k_j}}+ 
\cdots \cdots + 
\frac{a_{j,2}(\varphi)}{(\lambda +\lambda_j)^{2}}+ 
\frac{a_{j,1}(\varphi)}{(\lambda +\lambda_j)}+ 
\cdots \cdots  \qquad (a_{j,k}(\varphi) \in \RR) 
\end{equation}
be the Laurent expansion of $Z_f(\varphi)$ at 
$\lambda =-\lambda_j$. Then we obtain the 
following vanishing theorem 
which generalizes that of 
Jacobs in \cite{J}.

\begin{theorem}
Let $1 \leq k \leq k_j$. Assume that 
for any $\sigma \in \Sigma^{(k)}_j$ there exists 
$1 \leq i \leq k$ such that $\nu(\sigma)_i$ is 
odd. Then we have $a_{j,k}(\varphi)= \cdots = 
a_{j,k_j}(\varphi)=0$. 
\end{theorem} 
We can state also another vanishing theorem. 
In order to state it, let 
\begin{equation}
\varphi (x)=\sum_{\alpha \in \ZZ_+^n} 
c_{\alpha} x^{\alpha} \qquad 
(c_{\alpha} \in \RR) 
\end{equation}
be the Taylor expansion of the test 
function $\varphi$ at 
$0 \in U \subset \RR^n$. 
\begin{theorem}
For $1 \leq k \leq k_j$ 
assume that 
$\{ \alpha \in \ZZ_+^n \ | \ 
c_{\alpha} \not= 0\} \cap 
 \Delta_{j,k} = \emptyset$, 
where $\Delta_{j,k}$ is a certain compact 
subset of $\RR^n_+$ (see 
Definition \ref{DD}). 
Then we have 
$a_{j,k}(\varphi)= \cdots = 
a_{j,k_j}(\varphi)=0$. 
\end{theorem} 
Moreover, for the coefficients $a_{j,n}(\varphi)$ 
of the deepest poles $\lambda =-\lambda_j \in P$ 
we can prove the following explicit formula. 
For $\sigma \in \Sigma^{(n)}_j$ and $\alpha 
\in \ZZ_+^n$ we define an 
integer $\mu(\sigma, \alpha)_i$ by 
\begin{equation}
\mu(\sigma, \alpha)_i=\nu(\sigma)_i 
- \langle a^i(\sigma), \alpha \rangle 
\in \ZZ. 
\end{equation}

\begin{theorem}\label{TH3} 
Assume that $k_j=n$. 
Then $a_{j,n}(\varphi)$ is 
given by 
\begin{eqnarray}
 & a_{j,n}(\varphi) & 
\nonumber\\ 
 & = & \sum_{\alpha \in \Delta_{j,n}} 
  \left\{ 
\sum_{\sigma \in \Sigma^{(n)}_j} 
\left( \prod_{i=1}^n 
\frac{1+(-1)^{\nu(\sigma)_i}}{ 
l(a^i(\sigma)) \times \mu(\sigma , \alpha)_i!}
\right) 
\frac{\partial^{\mu(\sigma , \alpha)_1 + 
\cdots + \mu(\sigma , \alpha)_n} }{ 
\partial y_1^{\mu(\sigma , \alpha)_1} 
\cdots \partial y_n^{\mu(\sigma , \alpha)_n} }
 |f_{\sigma}|^{-\lambda_j} (0) \right\} 
\nonumber\\ 
 & \times & 
\frac{\partial_x^{\alpha} \varphi (0)}{\alpha !}, 
\end{eqnarray}
where $f_{\sigma}$ is a function defined in 
a neighborhood of $0 \in \RR^n_y$ such that 
$f_{\sigma}(0) \not= 0$ (see Section \ref{sec:3} 
for the definition) and 
for $\mu <0$ we set $\frac{\partial^{\mu}
}{\partial y_i^{\mu} } 
( \cdot )=0$. 
\end{theorem} 

Our results on the poles of local zeta function 
$Z_f(\varphi)$ also 
have some applications to 
oscillating 
integral $I_f(\varphi)(t)$ ($t \in \RR$) defined by 
\begin{equation}
I_f(\varphi)(t) =\int_{\RR^n} 
e^{it f(x)} \varphi (x)dx. 
\end{equation}
By the fundamental results of Varchenko \cite{V} 
(see also \cite{A-G-Z-V} for the detail), 
as $t \longrightarrow + \infty$ the oscillating 
integral $I_f(\varphi)(t)$ has an 
asymptotic expansion of the form 
\begin{equation} 
I_f(\varphi)(t) \sim \sum_{j=1}^{\infty} 
\sum_{k=1}^{k_j} c_{j,k}(\varphi) 
t^{-\lambda_j}(\log t)^{k-1}, 
\end{equation}
where $c_{j,k}(\varphi)$ are some complex 
numbers. Despite the important contributions 
by many mathematicians (see for example 
\cite{A-G-Z-V}, 
\cite{D-N-S}, \cite{G-S} and \cite{S} 
etc.), little 
is known about the coefficients $c_{j,k}(\varphi)$ 
of this asymptotic expansion. 
Here we can prove the following 
vanishing theorem. 

\begin{theorem}
\par \noindent (i) 
Let $1 \leq k \leq k_j$. 
Assume that $\lambda_j$ is not an integer and 
$\{ \alpha \in \ZZ_+^n \ | \ 
c_{\alpha} \not= 0\} \cap 
 \Delta_{j,k} = \emptyset$. 
Then we have $c_{j,k}(\varphi)= \cdots = 
c_{j,k_j}(\varphi)=0$. 
\par \noindent (ii) 
Let $2 \leq k \leq k_j$. 
Assume that $\lambda_j$ is an integer and 
$\{ \alpha \in \ZZ_+^n \ | \ 
c_{\alpha} \not= 0\} \cap 
 \Delta_{j,k-1} = \emptyset$. 
Then we have $c_{j,k}(\varphi)= \cdots = 
c_{j,k_j}(\varphi)=0$. 
\end{theorem} 
Moreover by Theorem \ref{TH3} we will 
prove also an explicit formula for 
the coefficients $c_{j,n}(\varphi)$ 
of $t^{-\lambda_j} (\log t)^{n-1}$ 
in the asymptotic expansion of 
$I_f(\varphi)$. See Section \ref{sec:5} 
for the detail. Finally let us mention 
that the method introduced in this 
paper would have some applications 
also in the study of 
$p$-adic local zeta functions 
(see \cite{I} etc.). It would 
be a very interesting subject to 
study the twisted monodromy 
conjecture (see \cite{N} for a 
review on this conjecture) 
by this method.

\section{Meromorphic continuations 
of distributions}\label{sec:2}

In this section, we prepare some basic results on 
the meromorphic continuations of the distributions 
$x_{1+}^{l_1\lambda +m_1} x_{2+}^{l_2\lambda +m_2} 
\cdots x_{n+}^{l_n\lambda +m_n} $ ($l_i \in \RR_{>0}$, 
$m_i \in \RR_{+}=\RR_{\geq 0}$) with respect to 
the complex parameter $\lambda$. In Section \ref{sec:3} 
and \ref{sec:4} these results will be used 
effectively to study 
the poles of local zeta functions. First, let 
us recall the classical result in the $1$-dimensional 
case (see \cite{G} etc.). Let $l \in \RR_{>0}$ 
be a positive real number and $m \in \RR_+$. Then 
for $\varphi \in C_0^{\infty}(\RR)$ the integral 
\begin{equation}
F_+(\varphi)(\lambda)=\int_{-\infty}^{+\infty} 
x_+^{l \lambda +m} \varphi (x) dx =\int_0^{+ \infty} 
x^{l \lambda +m} \varphi (x)  dx 
\end{equation} 
converges locally uniformly on $\{ \lambda \in \CC 
 \ | \ \Re \lambda > - \frac{m+1}{l} \}$ and defines 
a holomorphic function there. In other words, 
if $ \Re \lambda > - \frac{m+1}{l}$ the map 
$\varphi \mapsto F_+(\varphi)(\lambda) \in \RR$ 
defines a distribution $x_+^{l \lambda +m} \in 
\D^{\prime}(\RR)$ on $\RR$. Let us fix a test 
function $\varphi \in C_0^{\infty}(\RR)$. Following 
the methods in Gelfand-Shilov 
\cite{G-S} we shall extend 
$F_+(\varphi)$ to a meromorphic function 
on the whole complex plane $\CC$ as follows. 
First take a sufficiently large integer 
$N >>0$. Then for any $\lambda \in \CC$ 
such that $ \Re \lambda > - \frac{m+1}{l}$ 
we have 
\begin{eqnarray}
& & F_+(\varphi)(\lambda) 
=\int_0^{+ \infty} 
x^{l \lambda +m}\varphi (x) dx 
\nonumber\\
&=& \int_0^1 x^{l \lambda +m}
\left[ \varphi (x) - \sum_{r=1}^N 
\varphi^{(r-1)}(0) 
\frac{x^{r-1}}{(r-1)!}\right] dx 
\nonumber \\
& + & \int_1^{+ \infty} x^{l \lambda +m}\varphi (x) dx 
+ \sum_{r=1}^N \frac{\varphi^{(r-1)}(0)}{ 
(r-1)! (l \lambda +m+r)}
\nonumber \\
&=&\int_0^1 x^{l \lambda +m} dx 
\int_0^x \frac{\varphi^{(N)}(t)}{(N-1)!} 
(x-t)^{N-1}dt +\int_1^{+ \infty}
 x^{l \lambda +m}\varphi (x) dx
\nonumber \\
& + & \sum_{r=1}^N \frac{\varphi^{(r-1)}(0)}{ 
(r-1)! \times l (\lambda + \frac{m+r}{l})}
\nonumber\\
&=& \int_0^1 g_N(\lambda , t) \varphi^{(N)} 
(t) dt + \int_1^{+ \infty}
 t^{l \lambda +m}\varphi (t) dt
\nonumber\\
& + & \sum_{r=1}^N \frac{1}{ 
(r-1)! \times l (\lambda + \frac{m+r}{l})}
\langle (-1)^{r-1} \delta^{(r-1)},  
\varphi \rangle , 
\end{eqnarray}
where $\delta \in \D^{\prime}(\RR)$ is 
Dirac's delta function and we set 
\begin{equation} 
g_N(\lambda, t)=\frac{1}{(N-1)!} \int_t^1 
x^{l \lambda +m}(x-t)^{N-1}dx
\end{equation} 
for $0< t \leq 1$. The function 
$g_N(\lambda, t)$ satisfies the following 
nice properties. 

\begin{lemma}
\par \noindent (i) 
For any $0< t \leq 1$, $g_N(\lambda, t)$ 
is an entire function of $\lambda$. 
\par \noindent (ii) 
If $\Re \lambda > -\frac{m+N+1}{l}$ then 
the function $g_N(\lambda, t)$ of $t$ 
is integrable on $(0,1]$. 
\end{lemma}
By this lemma we see that the integral 
\begin{equation}
\int_0^1 g_N(\lambda , t) \varphi^{(N)} 
(t) dt 
\end{equation} 
converges locally uniformly on $\{ \lambda 
\in \CC \ | \ \Re \lambda > - \frac{m+N+1}{l} \}$ 
and defines a holomorphic function there. 
Since the integral $\int_1^{+ \infty}
 t^{l \lambda +m}\varphi (t) dt$ is 
an entire function of $\lambda$, 
the function $F_+(\varphi)$ is 
extended to a meromorphic function on 
$\{ \lambda 
\in \CC \ | \ \Re \lambda > 
- \frac{m+N+1}{l} \}$ with simple poles 
at $\lambda =- \frac{m+r}{l}$ ($r=1,2, 
\ldots , N$). Moreover the residue of 
$F_+(\varphi)$ at $\lambda =- \frac{m+r}{l}$ 
is given by 
\begin{equation}\label{R1} 
\text{Res}(F_+(\varphi) ; - \frac{m+r}{l})
=\frac{1}{(r-1)! \times l} 
\langle (-1)^{r-1} \delta^{(r-1)},  
\varphi \rangle . 
\end{equation} 
Similarly set 
\begin{equation}
F_-(\varphi)(\lambda)=\int_{-\infty}^{+\infty} 
x_{-}^{l \lambda +m} \varphi (x) dx 
=\int_{-\infty}^{0} 
 |x|^{l \lambda +m} \varphi (x)  dx.  
\end{equation} 
Then $F_-(\varphi)(\lambda)$ can be also 
extended to a meromorphic function on 
the complex plane $\CC$ with simple poles 
at $\lambda =- \frac{m+r}{l}$ ($r=1,2, 
\ldots , N$) and we have 
\begin{equation}\label{R2} 
\text{Res}(F_-(\varphi) ; 
- \frac{m+r}{l})
=\frac{1}{(r-1)! \times l} 
\langle \delta^{(r-1)},  
\varphi \rangle . 
\end{equation} 
This basic result in the $1$-dimensional case 
can be generalized to higher-dimensional cases 
as follows. Let $\varphi \in C_0^{\infty}(\RR^n)$ 
be a test function on $\RR^n$. 
For positive real numbers 
$l_1,l_2, \ldots , l_n \in \RR_{>0}$ and 
$m_1, m_2, \ldots , m_n \in \RR_+$ we set 
\begin{equation}
G(\varphi)(\lambda)=\int_{\RR^n} 
x_{1+}^{l_1 \lambda +m_1} x_{2+}^{l_2 \lambda +m_2} 
\cdots x_{n+}^{l_n \lambda +m_n} 
\varphi (x) dx
\end{equation}
and 
\begin{equation}
L=-\min \left\{ \frac{m_1+1}{l_1}, \frac{m_2+1}{l_2}, 
\cdots , \frac{m_n+1}{l_n} \right\} . 
\end{equation}
Then this integral 
converges locally uniformly 
on $\{ \lambda \in \CC \ | \ \Re \lambda >L \}$ 
and defines a holomorphic function there. 
By using the tensor product $\otimes$ of 
distributions we can rewrite $G(\varphi)(\lambda)$ as 
\begin{equation}\label{E1} 
G(\varphi)(\lambda)=
\langle  x_{1+}^{l_1 \lambda +m_1}
\otimes x_{2+}^{l_2 \lambda +m_2} \otimes 
\cdots \otimes x_{n+}^{l_n \lambda +m_n}, 
\varphi \rangle . 
\end{equation}
Let $N >>0$ be a sufficiently large integer. 
Then for $\Re \lambda >L$ we have the following 
equalities in the space $\D^{\prime}(\RR)$ of 
$1$-dimensional distributions. 
\begin{equation}
 x_{i+}^{l_i \lambda +m_i}
=G_{i,N}(\lambda)+\sum_{r=1}^N \frac{(-1)^{r-1}}{ 
(r-1)! l_i (\lambda +\frac{m_i+r}{l_i})} 
\delta^{(r-1)}(x_i) \qquad 
(i=1,2, \ldots , n), 
\end{equation}
where $G_{i,N}(\lambda) \in \D^{\prime}(\RR)$ 
is a $1$-dimensional distribution defined by 
\begin{equation}
 \langle G_{i,N}(\lambda), \phi \rangle 
=\int_0^1 g_{i,N}(\lambda , t) \phi^{(N)}(t) dt 
+\int_1^{+ \infty}t^{l_i\lambda +m_i} \phi(t)dt 
\qquad 
(\phi \in C_0^{\infty}(\RR)). 
\end{equation}
Here $g_{i,N}(\lambda , t)$ is an integrable function of 
$t \in (0, 1]$ for $\lambda \in \CC$ such that 
$\Re \lambda > \frac{m_i+N+1}{l_i}$. Putting 
this new expressions of the $1$-dimensional 
distributions $x_{i+}^{l_i \lambda +m_i}$ 
into \eqref{E1} we see that $G(\varphi)(\lambda)$ 
is extended to a meromorphic function on 
$\{ \lambda \in \CC \ | \ \Re \lambda > 
L_N\}$, where we set 
\begin{equation}
L_N=-\min \left\{ \frac{m_1+N+1}{l_1}, \frac{m_2+N+1}{l_2}, 
\cdots , \frac{m_n+N+1}{l_n} \right\} . 
\end{equation}
Since the integer $N>>0$ can be taken as 
large as possible, $G(\varphi)(\lambda)$ is 
meromorphically continued to the whole 
complex plane $\CC$. Moreover the poles of 
this meromorphic function $G(\varphi)(\lambda)$ 
are contained in a discrete set $P$ in $\CC$ 
defined by 
\begin{equation}
P=\bigcup_{1 \leq i \leq n}
\left\{ -\frac{m_i+1}{l_i}, -\frac{m_i+2}{l_i}, 
- \frac{m_i+3}{l_i}, \cdots \cdots \right\} 
\subset \RR_{<0} \subset \CC . 
\end{equation}
An element of $P$ is called a candidate pole of 
$G(\varphi)$. 
Let us rewrite this set $P$ as 
\begin{equation}
P=\left\{ -\lambda_1 > -\lambda_2 > -\lambda_3 > 
\cdots \cdots  \right\} \qquad 
(\lambda_j \in \RR_{>0}). 
\end{equation}
For each candidate pole $-\lambda_j \in P$ 
of $G(\varphi)$ we define a subset 
$S_j$ of $\{ 1,2, \ldots , n\}$ by 
\begin{equation}
S_j=\{ 1 \leq i \leq n \ | \ ^\exists r \in \ZZ_{>0} 
\ \text{such that} \ \frac{m_i+r}{l_i}=\lambda_j\} 
\end{equation}
and set $k_j =\sharp S_j$. Then we can easily see that 
the order of the pole of $G(\varphi)$ at 
$\lambda =-\lambda_j$ is less than or equal to $k_j$. 
For a candidate pole $-\lambda_j \in P$ 
of $G(\varphi)$ let 
\begin{equation}
\frac{a_{j,k_j}}{(\lambda +\lambda_j)^{k_j}}+ 
\cdots \cdots + 
\frac{a_{j,2}}{(\lambda +\lambda_j)^{2}}+ 
\frac{a_{j,1}}{(\lambda +\lambda_j)}+ 
\cdots \cdots  \qquad (a_{j,k} \in \RR) 
\end{equation}
be the Laurent expansion of $G(\varphi)$ at 
$\lambda =-\lambda_j$. For each $i \in S_j \subset 
\{ 1,2, \ldots , n\}$ we define a non-negative 
integer $\nu_i \in \ZZ_+$ by the formula 
\begin{equation}
\frac{m_i+1 + \nu_i}{l_i}=\lambda_j. 
\end{equation}
\begin{proposition}\label{KP} 
Let $1 \leq k \leq k_j$. Then  
the coefficient $a_{j,k}$ 
of $\frac{1}{(\lambda + \lambda_j)^k}$ in the 
Laurent expansion of $G(\varphi)(\lambda)$ at 
$\lambda =-\lambda_j$ is written as 
\begin{equation}
a_{j,k}=\sum_{S \subset S_j: \sharp S \geq k}
\frac{1}{(\sharp S -k)!} 
\left\{ 
\frac{\partial^{\sharp S-k}}{
 \partial \lambda^{\sharp S -k}} 
\rho_S(\lambda)\right\}_{\lambda =-\lambda_j}, 
\end{equation}
where for each subset $S \subset S_j$ of $S_j$ 
such that $\sharp S \geq k$ the function 
$\rho_S(\lambda)$ is holomorphic at 
$\lambda =-\lambda_j$ and written as follows: 
\par \noindent 
For the sake of simplicity, assume that 
$S=\{ 1,2, \ldots , l\}$ for some 
$l \geq k$. Then we have 
\begin{eqnarray}
\rho_S(\lambda) & = & 
\prod_{i=1}^l 
\frac{1}{l_i \nu_i !} \times 
\int_{ \{ x \in \RR^n | x_1= \cdots = x_l=0\} } 
\prod_{i=l+1}^n g_i(\lambda , x_i) 
\nonumber\\ 
 & & 
 \times 
\left\{ 
\frac{\partial^{\nu_1 + \cdots + 
 \nu_l}}{\partial x_1^{\nu_1} 
\cdots \partial x_l^{\nu_l}} 
\varphi (x) 
\right\}_{x_1= \cdots =x_l=0} 
  dx_{l+1} \cdots dx_n, 
\end{eqnarray}
where $g_i(\lambda , \cdot )$ 
($i=l+1, \ldots , n$) are $1$-dimensional 
integrable functions with holomorphic 
parameter $\lambda$ at 
$\lambda =-\lambda_j$. 
\end{proposition} 
\noindent When $k_j= \sharp S_j=n$ we have the following 
very simple expression of $a_{j,n}$. 
\begin{proposition}
If $k_j= \sharp S_j=n$, we have
\begin{equation}
a_{j,n}=\left( \prod_{i=1}^n 
\frac{1}{l_i\nu_i!} \right) 
\frac{\partial^{\nu_1+ \cdots +\nu_n}}{ 
\partial x_1^{\nu_1} \cdots 
\partial x_n^{\nu_n}} 
\varphi (0). 
\end{equation}
\end{proposition} 
\noindent Similarly let us set 
\begin{equation}
H(\varphi)(\lambda)=\int_{\RR^n} 
 |x_{1}|^{l_1 \lambda +m_1} |x_{2}|^{l_2 \lambda +m_2} 
\cdots |x_{n}|^{l_n \lambda +m_n} 
\varphi (x) dx
\end{equation}
Then $H(\varphi)$ can be also 
extended to a meromorphic function on 
$\CC$ whose poles are 
contained in the set $P 
\subset \RR_{<0}$. Moreover 
the order of the pole of $H 
(\varphi)$ at 
$\lambda =-\lambda_j \in P$ is less 
than or equal to $k_j$. 
For a candidate pole $-\lambda_j \in P$ 
of $H(\varphi)$ let 
\begin{equation}
\frac{b_{j,k_j}}{(\lambda +\lambda_j)^{k_j}}+ 
\cdots \cdots + 
\frac{b_{j,2}}{(\lambda +\lambda_j)^{2}}+ 
\frac{b_{j,1}}{(\lambda +\lambda_j)}+ 
\cdots \cdots  \qquad (b_{j,k} \in \RR) 
\end{equation}
be the Laurent expansion of $H
(\varphi)$ at 
$\lambda =-\lambda_j$. 
\begin{proposition}
Let $1 \leq k \leq k_j$. Then 
the coefficient $b_{j,k}$ 
of $\frac{1}{(\lambda + \lambda_j)^k}$ in the 
Laurent expansion of $H(\varphi)(\lambda)$ at 
$\lambda =-\lambda_j$ is written as 
\begin{equation}
b_{j,k}=\sum_{S \subset S_j: \sharp S \geq k}
\frac{1}{(\sharp S -k)!} 
\left\{ 
\frac{\partial^{\sharp S-k}}{
 \partial \lambda^{\sharp S -k}} 
\tau_S(\lambda)\right\}_{\lambda =-\lambda_j}, 
\end{equation}
where for each subset $S \subset S_j$ of $S_j$ 
such that $\sharp S \geq k$ the function 
$\tau_S(\lambda)$ is holomorphic at 
$\lambda =-\lambda_j$ and written as follows: 
\par \noindent 
For the sake of simplicity, assume that 
$S=\{ 1,2, \ldots , l\}$ for some 
$l \geq k$. Then we have 
\begin{eqnarray}
\tau_S(\lambda) & = & 
\prod_{i=1}^l 
\frac{1 +(-1)^{\nu_i}}{ l_i \nu_i !} \times 
\int_{ \{ x \in \RR^n  | 
 x_1= \cdots = x_l=0\} } 
\prod_{i=l+1}^n g_i(\lambda , x_i) 
\nonumber\\ 
 & & 
\times 
\left\{ 
\frac{\partial^{\nu_1 + \cdots + 
 \nu_l}}{\partial x_1^{\nu_1} 
\cdots \partial x_l^{\nu_l}} 
\varphi (x) 
\right\}_{x_1= \cdots =x_l=0} 
  dx_{l+1} \cdots dx_n, 
\end{eqnarray}
where $g_i(\lambda , \cdot )$ 
($i=l+1, \ldots , n$) are as 
in Proposition \ref{KP}. 
\end{proposition} 
As a special case of this proposition, 
we obtain the following. 
\begin{corollary}
Let $1 \leq k \leq k_j$. 
Assume that for any $S \subset S_j$ 
with $\sharp S =k$ there exists 
$i \in S$ such that $\nu_i$ is odd. 
Then we have $b_{j,k}= \cdots =b_{j,k_j}=0$. 
\end{corollary} 
If $k_j= \sharp S_j=n$, we 
can also obtain the following 
explicit expression of $b_{j,n}$. 
\begin{proposition}
If $k_j= \sharp S_j=n$, we have
\begin{equation}
b_{j,n}=\left( \prod_{i=1}^n
 \frac{1+(-1)^{\nu_i} 
}{l_i\nu_i! } \right) 
\frac{\partial^{\nu_1+ \cdots +\nu_n}}{ 
\partial x_1^{\nu_1} \cdots 
\partial x_n^{\nu_n}} 
\varphi (0). 
\end{equation}
\end{proposition}

\section{Vanishing theorems for the poles of 
local zeta functions}\label{sec:3}

Let $f$ be a real-valued real analytic function 
defined on an open neighborhood $U$ of $0 \in \RR^n$. 
Then for any real-valued test function 
$\varphi \in C_0^{\infty}(U)$ the integral 
\begin{equation}
Z_f(\varphi)(\lambda) =\int_{\RR^n} 
 |f(x)|^{\lambda} \varphi (x)dx 
\end{equation}
converges locally uniformly 
on $\{ \lambda \in \CC \ | \ 
\Re \lambda >0\}$ and defines a holomorphic 
function there. Moreover it is well-known 
that $Z_f(\varphi)(\lambda)$ can be extended 
to a meromorphic function defined on the whole 
complex plane $\CC$ (see 
\cite{A-G-Z-V}, \cite{Kan} and \cite{V} etc.). 
In this section, assuming that the real hypersurface 
$\{ x \in U \ | \ f(x)=0 \}$ has an isolated 
singular point at $0 \in U \subset \RR^n$, 
we prove some general vanishing theorems on 
the poles of the local zeta function 
$Z_f(\varphi)(\lambda)$.  First, let 
\begin{equation}
f(x)=\sum_{\alpha \in \ZZ_+^n} a_{\alpha} x^{\alpha} 
\qquad (a_{\alpha} \in \RR) 
\end{equation}
be the Taylor expansion of $f$ at $0 \in \RR^n$. 
\begin{definition}
\par \noindent (i) 
Let $\Gamma_+(f) \subset \RR_+^n$ be the convex hull 
of $\bigcup_{\alpha : a_{\alpha} \not= 0} ( \alpha + 
\RR_+^n)$ in $\RR_+^n=\RR^n_{\geq 0}$. We call 
$\Gamma_+(f)$ the Newton polygon of $f$. 
\par \noindent (ii) 
For each compact face $\gamma \prec \Gamma_+(f)$ 
of $\Gamma_+(f)$ we set 
\begin{equation}
f_{\gamma}(x)=\sum_{\alpha \in \gamma \cap \ZZ_+^n} 
a_{\alpha} x^{\alpha} 
\in \RR [x_1, x_2, \ldots , x_n]. 
\end{equation}
We call $f_{\gamma}(x)$ the $\gamma$-part of $f$. 
\end{definition}
From now on, we assume that $f$ satisfies 
the following condition. 
\begin{definition}\label{ND} 
We say that $f$ is non-degenerate if for any 
compact face $\gamma \prec \Gamma_+(f)$ 
of $\Gamma_+(f)$ we have: 
\begin{equation}
\left( \frac{\partial f_{\gamma}}{\partial x_1}(x), 
\frac{\partial f_{\gamma}}{\partial x_2}(x), 
\cdots \cdots , 
\frac{\partial f_{\gamma}}{\partial x_n}(x) \right) 
\not= (0,0, \cdots \cdots , 0)
\end{equation}
at any point $x$ of $\{ x\in \RR^n \ | \ 
x_1x_2 \cdots x_n \not= 0, f_{\gamma}(x)=0 \} 
\subset (\RR \setminus \{ 0\} )^n$. 
\end{definition}
For the sake of simplicity, let us assume also 
that $f$ is convenient: for any 
$1 \leq i \leq n$ we have 
$\Gamma_+(f) \cap 
\{ \alpha \in \RR_+^n \ | \ 
\alpha_j=0 \ \text{for} \ j \not= i \ 
\text{and} \ \alpha_i>0 \} \not= \emptyset$. 
Then by the results of Varchenko \cite{V} 
(see also \cite{A-G-Z-V} and 
\cite{Kan}) we can describe 
the candidate poles of $Z_f(\varphi)$ 
in terms of $\Gamma_+(f)$ 
as follows. First Let $\Sigma_0$ be the 
dual fan of $\Gamma_+(f)$ and $\Sigma$ 
a subdivision of $\Sigma_0$ such that the 
real toric variety $X_{\Sigma}$ associated 
to it is smooth. Since $\Sigma =\{ \sigma \}$ 
is a family of rational polyhedral convex 
cones in $\RR_+^n$ such that $\RR_+^n= 
\cup_{\sigma \in \Sigma}\sigma$ there exists 
a natural proper 
morphism $\pi : X_{\Sigma} \longrightarrow 
\RR^n$ of real analytic manifolds. 
By using the convenience of $f$, we can 
construct a smooth fan 
$\Sigma$ such that $\pi$ induces a diffeomorphism 
$X_{\Sigma} \setminus \pi^{-1}(0) \simto 
\RR^n \setminus \{ 0\} $. Moreover by the 
non-degeneracy of $f$, the pull-back 
$f \circ \pi : X_{\Sigma} \longrightarrow \RR$ 
of $f$ to $X_{\Sigma}$ 
defines a hypersurface $\{ f \circ \pi =0\}$ 
in $X_{\Sigma}$ with 
only normal crossing singularities. 
In order to describe 
$\pi : X_{\Sigma} \longrightarrow 
\RR^n$ and $f \circ \pi $ 
we shall prepare some notations. 
By the smoothness of $X_{\Sigma}$, on 
any cone $\sigma \in \Sigma$ such that 
$\d \sigma =k$ there exist 
exactly $k$ edges (i.e. $1$-dimensional 
faces). Let $a^1(\sigma), a^2(\sigma), 
\ldots , a^k(\sigma) \in \partial \sigma 
\cap (\ZZ^n \setminus \{ 0\} )$ be 
the non-zero primitive 
vectors on these edges of $\sigma$. 
We call $\{  a^1(\sigma), a^2(\sigma), 
\ldots , a^k(\sigma)  \}$ the $1$-skelton of 
$\sigma$. For each $n$-dimensional cone 
$\sigma \in \Sigma$ we fix 
the ordering 
of its $1$-skelton $\{  a^1(\sigma), a^2(\sigma), 
\ldots , a^n(\sigma)  \}$ so that the 
determinant of the invertible matrix 
\begin{equation}
A(\sigma)=\{ a^i(\sigma)_j \}_{i,j=1}^n 
\in GL_n(\ZZ)
\end{equation}
is $1$. For a cone $\sigma \in \Sigma$ and 
$1 \leq i \leq \d \sigma$ we set 
\begin{equation}
l(a^i(\sigma))=\min_{\alpha \in \Gamma_+(f)} 
\langle a^i(\sigma), \alpha \rangle \in \ZZ_+ 
\end{equation}
and 
\begin{equation}
 |a^i(\sigma)| =\sum_{j=1}^n a^i(\sigma)_j 
\in \ZZ_{>0}. 
\end{equation}
Now let $\sigma$ be 
an $n$-dimensional cone in $\Sigma$ 
and $\RR^n(\sigma) \simeq \RR^n_y$ 
 the affine open subset of $X_{\Sigma}$ 
associated to $\sigma$. Then the restriction 
$\pi (\sigma): \RR^n(\sigma) 
\longrightarrow \RR^n$ of 
$\pi : X_{\Sigma} \longrightarrow 
\RR^n$ to $\RR^n(\sigma) \simeq \RR^n_y$ 
and its Jacobian $J(\pi (\sigma)): \RR^n(\sigma) 
\longrightarrow \RR$ 
are explicitly given by 
\begin{equation}
\pi (\sigma)(y)=\left( 
\prod_{i=1}^n y_i^{a^i(\sigma)_1}, 
\prod_{i=1}^n y_i^{a^i(\sigma)_2}, 
\cdots , 
\prod_{i=1}^n y_i^{a^i(\sigma)_n}
\right),
\end{equation}
\begin{equation}
J(\pi (\sigma))(y)=
y_1^{|a^1(\sigma)|-1} y_2^{|a^2(\sigma)|-1} 
\cdots y_n^{|a^n(\sigma)|-1}. 
\end{equation}
Moreover on $\RR^n(\sigma) \simeq \RR^n_y$ 
we have 
\begin{equation}
(f \circ \pi (\sigma))(y)=
f_{\sigma}(y) \times \prod_{i=1}^n 
y_i^{l(a^i(\sigma))}, 
\end{equation}
where $f_{\sigma}$ is a real analytic 
function defined on $\pi (\sigma)^{-1}(U) 
\subset \RR^n(\sigma)$. By the non-degeneracy 
of $f$ the hypersurface $\{ f_{\sigma}=0\}$ 
intersect all coordinate subspaces 
of $\RR^n(\sigma)$ transversally. In 
particular, we have $f_{\sigma}(0)\not= 0$. 
For $0 \leq k \leq n$ let $\Sigma^{(k)} 
\subset \Sigma$ be the subset of 
$\Sigma$ consisting of $k$-dimensional 
cones. 
\begin{definition}
(see \cite{A-G-Z-V}, \cite{Kan} 
 and \cite{V} etc.) 
Let $P \subset \QQ_{<0}$ be the union of 
the following subsets of $\QQ_{<0}$: 
\begin{equation}
\{ -1,-2, -3, \cdots \cdots \}, 
\end{equation}
\begin{equation}
\left\{ -\frac{|a^1(\sigma)|}{l(a^1(\sigma))}, 
-\frac{|a^1(\sigma)|+1}{l(a^1(\sigma))}, 
\cdots \cdots \right\} 
\qquad 
(\sigma \in \Sigma^{(1)} 
 \ \text{such that} \ 
 l(a^1(\sigma))>0). 
\end{equation}
\end{definition} 
By the results of \cite{V}, the poles of the 
local zeta function $Z_f(\varphi)$ are contained 
in the set $P$. An element of $P$ is called a 
candidate pole of $Z_f(\varphi)$. We order 
the candidate poles of $Z_f(\varphi)$ as 
\begin{equation}
P= 
\{ -\lambda_1 > -\lambda_2 > - \lambda_3 > 
 \cdots \cdots \} \qquad (\lambda_j \in \QQ_{>0}). 
\end{equation}
Hereafter we fix a candidate pole $-\lambda_j 
\in P$ of $Z_f(\varphi)$. 
\begin{definition}
\par \noindent (i) 
Let $\sigma \in \Sigma$. For $1 \leq i \leq \d \sigma$ 
such that $l(a^i(\sigma))>0$ ($\Longleftrightarrow 
a^i(\sigma) \in \RR_{>0}^n$) we set 
\begin{equation}
K^i(\sigma)=\left\{ 
\frac{|a^i(\sigma)|}{l(a^i(\sigma))}, 
\frac{|a^i(\sigma)|+1}{l(a^i(\sigma))}, 
\cdots \cdots \right\} \subset \QQ_{>0}. 
\end{equation}
\par \noindent (ii) 
For the candidate pole $-\lambda_j \in P$ 
of $Z_f(\varphi)$ and $0 \leq k \leq n$ we 
define a subset $\Sigma^{(k)}_j$ of $\Sigma^{(k)}$ 
by 
\begin{equation}
\Sigma^{(k)}_j=\{ \sigma \in \Sigma^{(k)} \ | \ 
l(a^i(\sigma))>0 \ \text{and} \ 
\lambda_j \in K^i(\sigma) 
\ \text{for} \ 1 \leq i \leq k \}. 
\end{equation}
\par \noindent (iii) 
For $\sigma \in \Sigma^{(k)}_j$ and $1 \leq i \leq k$ 
we define a non-negative integer $\nu(\sigma)_i 
\in \ZZ_+$ by 
\begin{equation}
\lambda_j=\frac{ |a^i(\sigma)| + \nu(\sigma)_i}{ 
l(a^i(\sigma))}. 
\end{equation}
\end{definition}
After \cite{A-G-Z-V} and \cite{V} the following 
result is well-known to the specialists. 
\begin{theorem}
\par \noindent (i) 
Assume that $\lambda_j \notin \ZZ$. Then 
the order of the pole of $Z(\varphi)$ at 
$\lambda =-\lambda_j$ is less than or equal to 
\begin{equation}
k_j:= \max \{ 0 \leq k \leq n \ | \ 
\Sigma_j^{(k)} \not= \emptyset \} \in \ZZ_+. 
\end{equation}
\par \noindent (ii) 
Assume that $\lambda_j \in \ZZ$. Then 
the order of the pole of $Z(\varphi)$ at 
$\lambda =-\lambda_j$ is less than or equal to 
\begin{equation}
k_j:= 1+ \max \{ 0 \leq k \leq n \ | \ 
\Sigma_j^{(k)} \not= \emptyset \} \in \ZZ_+. 
\end{equation}
\end{theorem} 
\noindent Now let 
\begin{equation}
\frac{a_{j,k_j}(\varphi) }{(\lambda +\lambda_j)^{k_j}}+ 
\cdots \cdots + 
\frac{a_{j,2}(\varphi)}{(\lambda +\lambda_j)^{2}}+ 
\frac{a_{j,1}(\varphi)}{(\lambda +\lambda_j)}+ 
\cdots \cdots  \qquad (a_{j,k}(\varphi) \in \RR) 
\end{equation}
be the Laurent expansion of $Z_f(\varphi)$ at 
$\lambda =-\lambda_j$. Then we obtain the 
following result which generalizes that of 
Jacobs in \cite{J}. 
\begin{theorem}\label{MT} 
\par \noindent (i) 
Assume that $\lambda_j$ is not an odd integer and 
let $1 \leq k \leq k_j$. Assume moreover that 
for any $\sigma \in \Sigma^{(k)}_j$ there exists 
$1 \leq i \leq k$ such that $\nu(\sigma)_i$ is 
odd. Then we have $a_{j,k}(\varphi)= \cdots = 
a_{j,k_j}(\varphi)=0$. 
\par \noindent (ii) 
Assume that $\lambda_j$ is an odd integer and 
let $2 \leq k \leq k_j$. Assume moreover that 
for any $\sigma \in \Sigma^{(k-1)}_j$ there exists 
$1 \leq i \leq k-1$ such that $\nu(\sigma)_i$ is 
odd. Then we have $a_{j,k}(\varphi)= \cdots = 
a_{j,k_j}(\varphi)=0$. 
\end{theorem} 

\begin{proof}
(i) Since $\supp \varphi$ is compact and $\pi : 
X_{\Sigma} \longrightarrow \RR^n$ is proper, 
there exists finite $C^{\infty}$-functions 
$\varphi_q$ ($1 \leq q \leq N$) on $X_{\Sigma}$ 
such that $\sum_{q=1}^N \varphi_q \equiv 1$ on 
$\supp (\varphi \circ \pi)$. We may assume 
that for any $1 \leq q \leq N$ there exists 
an $n$-dimensional cone $\sigma_q \in \Sigma^{(n)}$ 
such that $\supp \varphi_q 
\subset \subset \RR^n(\sigma_q)$. Then we have 
\begin{eqnarray}
& & Z_f(\varphi)(\lambda)= 
\nonumber\\
& = & 
\sum_{q=1}^N \int_{\RR^n(\sigma_q)} 
\prod_{i=1}^n |y_i|^{l(a^i(\sigma_q))\lambda +
|a^i(\sigma_q)|-1} \times 
 |f_{\sigma_q}|^{\lambda} (y) \times 
(\varphi \circ \pi(\sigma_q))(y)
 \varphi_q(y) dy.  
\end{eqnarray}
We divide the proof of (i) into the 
following two case. 
\medskip 
\par 
\noindent (I) First assume that $\lambda_j$ 
is not an integer. Then by (the proof of) 
Proposition \ref{KP}, $a_{j,k}(\varphi)$ can 
be written as a
\begin{equation}
a_{j,k}(\varphi)= 
\sum_{q=1}^N \sum_{l \geq k} 
\sum_{\sigma \in \Sigma^{(l)}_{j,q}} 
J_q(\sigma), 
\end{equation} 
where for $1 \leq q \leq N$ and $0 \leq l \leq n$ 
we set 
\begin{equation}
\Sigma^{(l)}_{j,q}=\{ \sigma \in \Sigma^{(l)}_j 
 \ | \ \sigma \prec \sigma_q \}. 
\end{equation} 
Moreover for $l$ such that $k \leq l \leq n$ 
and $\sigma \in 
\Sigma^{(l)}_{j,q}$ the number $J_q(\sigma)$ 
is expressed as follows. 
\begin{equation}
J_q(\sigma)=\frac{1}{(l-k)!} \times 
\frac{d^{l-k}}{d \lambda^{l-k}}
 \rho_{q, \sigma}(\lambda) |_{\lambda =-\lambda_j}. 
\end{equation} 
Let us explain the 
function $\rho_{q, \sigma}(\lambda)$ 
which is holomorphic at $\lambda =-\lambda_j$. 
For the sake of 
simplicity, we assume that $\{ a^1(\sigma_q), 
a^2(\sigma_q), \ldots , a^l(\sigma_q)\}$ 
is the $1$-skelton of $\sigma \prec \sigma_q$. 
\medskip 
\par 
\noindent (a) (The case where $\supp \varphi_q 
\cap \{ y \in \RR^n(\sigma_q) \ | \ 
f_{\sigma_q}(y)= 
y_1=y_2= \cdots =y_l=0 \} = \emptyset$) 
We set 
\begin{eqnarray}
 & \rho_{q, \sigma}(\lambda) & = 
 \prod_{i=1}^l 
\frac{1+(-1)^{\nu(\sigma_q)_i}}{ 
l(a^i(\sigma_q))  \nu(\sigma_q)_i !}
\int_{ \{ y_1= \cdots = y_l=0\} } 
\prod_{i=l+1}^n g_i(\lambda , y_i) 
\nonumber\\ 
 & \times & 
\frac{\partial^{\nu(\sigma_q)_1 + \cdots + 
 \nu(\sigma_q)_l}}{\partial y_1^{\nu(\sigma_q)_1} 
\cdots \partial y_l^{\nu(\sigma_q)_l}} 
 \left\{ 
 |f_{\sigma_q}|^{\lambda} (\varphi \circ 
\pi (\sigma_q)) \varphi_q
 \right\}_{y_1= \cdots =y_l=0} 
  dy_{l+1} \cdots dy_n, 
\end{eqnarray}
where $g_{l+1}(\lambda , \cdot), \ldots , 
g_n(\lambda , \cdot )$ are $1$-dimensional 
integrable functions with holomorphic parameter 
$\lambda$ at $\lambda
 =-\lambda_j \in P$. 
\medskip 
\par 
\noindent (b) (The case where $\supp \varphi_q 
\cap \{ y \in \RR^n(\sigma_q) \ | \ 
f_{\sigma_q}(y)=  
y_1=y_2= \cdots =y_l=0 \} \not= \emptyset$) 
For $1 \leq i \leq n$ set $H_i=\{ 
y\in \RR^n(\sigma_q) \ | \ y_i=0\}$. 
Then we may assume that $\{ 
1 \leq i \leq n \ | \ \supp \varphi_q 
\cap H_i \not= \emptyset \} =\{ 
1,2, \ldots , r\}$ for some 
$r \geq l$. In this 
case, by a real analytic local coordinate 
change $\Phi : (y_1, \ldots , y_n) 
\longmapsto (z_1, \ldots , z_n)$ 
such that $z_i=y_i$ ($1 \leq i \leq r$) 
which sends the hypersurface $\{ 
f_{\sigma_q}=0\}$ to $\{ z_n=0 \}$, the 
function $\rho_{q, \sigma}(\lambda)$ 
is expressed as 
\begin{eqnarray}
 & \rho_{q, \sigma}(\lambda) & =  
 \left( \prod_{i=1}^l 
\frac{1+(-1)^{\nu(\sigma_q)_i}}{ 
l(a^i(\sigma_q))  \nu(\sigma_q)_i !} \right) 
\int_{ \{ z_1= \cdots z_l=0\} } 
\left( 
\prod_{i=l+1}^r g_i(\lambda , z_i) 
\right)  g_n(\lambda , z_n) 
\nonumber\\ 
  &  \frac{\partial^{\nu(\sigma_q)_1 + \cdots + 
 \nu(\sigma_q)_l}}{ \partial z_1^{\nu(\sigma_q)_1} 
\cdots \partial z_l^{\nu(\sigma_q)_l}} & 
\left\{  F(\lambda , z) 
 (\varphi \circ \pi (\sigma_q) \circ \Phi^{-1}) 
(\varphi_q \circ \Phi^{-1}) 
| \frac{\partial (y_1, \ldots , y_n)}{ 
\partial (z_1, \ldots , z_n)} |
\right\}_{z_1= \cdots =z_l=0}
\nonumber\\ 
 & &  dz_{l+1} \cdots dz_n, 
\end{eqnarray}
where 
\begin{equation}
F(\lambda , z) = 
\left( \prod_{i=r+1}^n 
 |y_i|^{l(a^i(\sigma_q))\lambda + 
 |a^i(\sigma_q)|-1} \right) \circ \Phi^{-1} 
\end{equation}
is a real analytic function on a 
neighborhood of $\Phi (\supp \varphi_q)$ 
and $g_{l+1}(\lambda , \cdot)$, $\ldots$, 
$g_r(\lambda , \cdot )$ and $g_n(\lambda , \cdot )$ 
are $1$-dimensional 
integrable functions with holomorphic parameter 
$\lambda$ at $\lambda
 =-\lambda_j \in P$. 
\medskip 
\par 
\noindent Now by our assumption, for 
any $\sigma \in \Sigma_{j,q}^{(l)}$ 
($l \geq k$) there exists $1 \leq i \leq l$ 
such that $\nu(\sigma_q)_i$ is odd. Therefore we 
obtain $a_{j,k}(\varphi)=
0$ in this case. In the same way, we can prove 
also that $a_{j,k+1}(\varphi)= 
\cdots a_{j,k_j}(\varphi)=0$. 
\medskip 
\par 
\noindent (II) Next assume that $\lambda_j$ 
is an integer and set $m:=\lambda_j$. 
Then by (the proof of) 
Proposition \ref{KP}, $a_{j,k}(\varphi)$ can 
be written as 
\begin{equation}
\sum_{q=1}^N \left\{ 
\sum_{l \geq k} 
\sum_{\sigma \in \Sigma^{(l)}_{j,q}} 
J_q(\sigma) + 
\sum_{l \geq k-1} 
\sum_{\sigma \in \Sigma^{(l)}_{j,q}} 
\tilde{J}_q(\sigma) \right\} , 
\end{equation} 
where for $l$ such that $k-1 \leq l \leq n$ 
and $\sigma \in 
\Sigma^{(l)}_{j,q}$ the number 
$\tilde{J}_q(\sigma)$ 
is expressed as follows. 
\begin{equation}
\tilde{J}_q(\sigma)=\frac{1}{(l+1-k)!} \times 
\frac{d^{l+1-k}}{d \lambda^{l+1-k}}
 \tau_{q, \sigma}(\lambda) |_{\lambda =-\lambda_j}. 
\end{equation} 
Let us explain the 
function $\tau_{q, \sigma}(\lambda)$ 
which is holomorphic at $\lambda =-\lambda_j$. 
For simplicity, we assume that $\{ a^1(\sigma_q), 
a^2(\sigma_q), \ldots , a^l(\sigma_q)\}$ 
is the $1$-skelton of 
$\sigma \prec \sigma_q$. If $\supp \varphi_q 
\cap \{ y \in \RR^n(\sigma_q) \ | \ 
f_{\sigma_q}(y)= 
y_1=y_2= \cdots =y_l=0 \} = \emptyset$ 
we set $\tau_{q, \sigma}(\lambda) \equiv 0$. 
If $\supp \varphi_q 
\cap \{ y \in \RR^n(\sigma_q) \ | \ 
f_{\sigma_q}(y)= 
y_1=y_2= \cdots =y_l=0 \} \not= \emptyset$, 
assuming as in (b) that $\{ 
1 \leq i \leq n \ | \ \supp \varphi_q 
\cap H_i \not= \emptyset \} =\{ 
1,2, \ldots , r\}$  for some $r \geq l$ 
and using the local 
coordinate 
change $\Phi$ used in (b), the 
function $\tau_{q, \sigma}(\lambda)$ is expressed as 
\begin{eqnarray}
 & \tau_{q, \sigma}(\lambda) & = 
\left(  
\prod_{i=1}^l 
\frac{1+(-1)^{\nu(\sigma_q)_i}}{ 
l(a^i(\sigma_q))  \nu(\sigma_q)_i !}
 \right) 
\frac{ 1+(-1)^{m-1}}{(m-1)!} 
\int_{ \{ z_1= \cdots =z_l=
z_n=0 \} } 
\left( 
\prod_{i=l+1}^r g_i(\lambda , z_i) 
\right) 
\nonumber\\ 
  & \frac{\partial^{\nu(\sigma_q)_1 + \cdots + 
 \nu(\sigma_q)_l +m-1}}{ 
\partial z_1^{\nu(\sigma_q)_1} 
\cdots \partial 
z_l^{\nu(\sigma_q)_l} \partial 
z_n^{m-1}} & 
 \left\{  F(\lambda , z) 
 (\varphi \circ \pi (\sigma_q) \circ \Phi^{-1}) 
(\varphi_q \circ \Phi^{-1}) 
| \frac{\partial (y_1, \ldots , y_n)}{ 
\partial (z_1, \ldots , z_n)} | 
 \right\}_{z_1= \cdots =z_l=z_n=0} 
\nonumber\\ 
 & & dz_{l+1} \cdots dz_{n-1}, 
\end{eqnarray}
where $F(\lambda , z)$ and 
$g_{l+1}(\lambda , \cdot), \ldots , 
g_r(\lambda , \cdot )$ 
are as in (b). Then, as in 
(I), by our assumption 
we obatin $J_q(\sigma)=0$ for 
any $\sigma \in \Sigma_{j,q}^{(l)}$ ($l \geq k$). 
Moreover, since by our assumption 
in (i) the integer $m=\lambda_j$ 
must be even, 
we obtain $\tilde{J}_q 
(\sigma)=0$ for 
any $\sigma \in \Sigma_{j,q}^{(l)}$ 
($l \geq k-1$). Therefore we get 
$a_{j,k}(\varphi)=0$ in this case, 
too. In the same way, we can prove 
also that $a_{j,k+1}(\varphi)= 
\cdots a_{j,k_j}(\varphi)=0$. 
This completes the proof of (i). 
The remaining assertion (ii) can be 
shown similarly. 
\end{proof}

By this theorem we see that many candidate poles of 
$Z_f(\varphi)$ are fake, i.e. not actual poles. 
We can also find a nice condition on the test 
function $\varphi \in C_0^{\infty}(\RR^n)$ 
under which we have the vanishing 
$a_{j,k}(\varphi)= \cdots = 
a_{j,k_j}(\varphi)=0$. For this purpose, 
we introduce the following subset $\Delta_{j,k}$ 
of $\RR^n_+$. 
\begin{definition}\label{DD} 
Let $1 \leq k \leq k_j$. 
\par \noindent (i) 
For $\sigma \in \Sigma^{(k)}_j$ we define a 
compact convex subset  $\Delta_{j,k}^{\sigma}$ 
of $\RR^n_+$ by 
\begin{equation}
\Delta_{j,k}^{\sigma}=\{ \alpha \in \RR_+^n \ | \ 
\langle a^i(\sigma), \alpha \rangle 
\leq \nu(\sigma)_i \quad \text{for any} 
\quad 1 \leq i \leq k \}. 
\end{equation}
\par \noindent (ii) 
We define a compact subset $\Delta_{j,k}$ 
of $\RR^n_+$ by 
\begin{equation}
\Delta_{j,k}= \bigcup_{\sigma \in 
\Sigma^{(k)}_j } \Delta_{j,k}^{\sigma}. 
\end{equation}
\end{definition} 
Note that $\Delta_{j,k}$ is not necessarily a  
convex subset of $\RR^n_+$. It follows also 
from the definition that we have 
$\Delta_{j,k} \supset \Delta_{j,k^{\prime}}$ 
for $1 \leq k \leq k^{\prime} \leq k_j$. 
In order to state our 
another vanishing theorem, let 
\begin{equation}
\varphi (x)=\sum_{\alpha \in \ZZ_+^n} 
c_{\alpha} x^{\alpha} \qquad 
(c_{\alpha} \in \RR) 
\end{equation}
be the Taylor expansion of the test 
function $\varphi$ at $0 \in U \subset \RR^n$. 
\begin{theorem}\label{VT} 
\par \noindent (i) 
Let $1 \leq k \leq k_j$. 
Assume that $\lambda_j$ is not an odd integer and 
$\{ \alpha \in \ZZ_+^n \ | \ 
c_{\alpha} \not= 0\} \cap 
 \Delta_{j,k} = \emptyset$. 
Then we have $a_{j,k}(\varphi)= \cdots = 
a_{j,k_j}(\varphi)=0$. 
\par \noindent (ii) 
Let $2 \leq k \leq k_j$. 
Assume that $\lambda_j$ is an odd integer and 
$\{ \alpha \in \ZZ_+^n \ | \ 
c_{\alpha} \not= 0\} \cap 
 \Delta_{j,k-1} = \emptyset$. 
Then we have $a_{j,k}(\varphi)= \cdots = 
a_{j,k_j}(\varphi)=0$. 
\end{theorem} 

\begin{proof}
We use the notations in the proof of 
Theorem \ref{MT}. 
\par 
\noindent (i) Since the proof for the case 
where $\lambda_j$ is an integer is 
similar, we prove the assertion 
only in the case where $\lambda_j$ is 
not an integer. In this case, we have 
\begin{equation}
a_{j,k}(\varphi)= 
\sum_{q=1}^N \sum_{l \geq k} 
\sum_{\sigma \in \Sigma^{(l)}_{j,q}} 
\frac{1}{(l-k)!} 
\frac{d^{l-k}}{d \lambda^{l-k}}
 \rho_{q, \sigma}(\lambda) |_{\lambda =-\lambda_j}, 
\end{equation} 
where the function $\rho_{q, \sigma}(\lambda)$ is 
holomorphic at $\lambda =-\lambda_j$ 
(for its expression, 
see the proof of Theorem \ref{MT}). 
For $\alpha \in \ZZ_+^n$ let $\psi_{\alpha} 
\in C_0^{\infty}(U)$ be a test function 
on $U$ such that $\psi_{\alpha} \equiv 
x^{\alpha}$ in a neighborhood of $0 \in 
U \subset \RR^n$. For $1 \leq q \leq N$ 
and $\sigma \in \Sigma_{j,q}^{(l)}$ 
($l \geq k$), assume for simplicity 
that $\{ a^1(\sigma_q), 
a^2(\sigma_q), \ldots , a^l(\sigma_q)\}$ 
is the $1$-skelton of 
$\sigma \prec \sigma_q$. Then in an 
open neighborhood of 
$\{ y \in \RR^n(\sigma_q) \ | \ 
y_1= \cdots =y_l=0\} \subset \RR^n(\sigma_q) 
\subset X_{\Sigma}$ we have 
\begin{equation} 
(\psi_{\alpha} \circ \pi (\sigma_q))(y) \equiv 
y_1^{\langle a^1(\sigma_q), \alpha \rangle} 
\cdots 
y_n^{\langle a^n(\sigma_q), \alpha \rangle}.  
\end{equation} 
Moreover, by the definition of $\Delta_{j,k}$, 
if $\alpha \notin \Delta_{j,k}$ then 
we have $\alpha \notin \Delta_{j,l}$ and 
there exists $1 \leq i \leq l$ such that 
$\langle a^i(\sigma_q), \alpha \rangle > 
\nu(\sigma_q)_i$. This implies 
the vanishing of the function 
\begin{equation} 
\frac{\partial^{\nu(\sigma_q)_1 + \cdots + 
 \nu(\sigma_q)_l}}{\partial y_1^{\nu(\sigma_q)_1} 
\cdots \partial y_l^{\nu(\sigma_q)_l}} 
\left\{ 
 |f_{\sigma_q}|^{\lambda} (\psi_{\alpha} \circ 
\pi (\sigma_q)) \varphi_q
 \right\}_{y_1= \cdots =y_l=0} \equiv 0. 
\end{equation}
By applying this vanishing result to the 
expression of $\rho_{q, \sigma}(\lambda)$ 
(in the proof of Theorem \ref{MT}), 
we see that if $\{ \alpha \in \ZZ_+^n \ | \ 
c_{\alpha} \not= 0\} \cap \Delta_{j,k} 
= \emptyset$ the vanishing $a_{j,k}(\varphi)=0$ 
holds. In the same way, we can prove 
also that $a_{j,k+1}(\varphi)= 
\cdots a_{j,k_j}(\varphi)=0$. 
This completes the proof of (i). 
The assertion (ii) can be shown similarly. 
\end{proof} 

Now let us consider the following two 
local zeta functions. 
\begin{equation}
Z_f^{\pm}(\varphi)(\lambda) =\int_{\RR^n} 
 f_{\pm}^{\lambda} (x) \varphi (x)dx. 
\end{equation}
Note that we have $Z_f(\varphi)
=Z_f^{+}(\varphi)+ Z_f^{-}(\varphi)$. 
Then the poles of these lcoal zeta functions 
$Z_f^{\pm}(\varphi)$ are also contained in $P$ 
and their Laurent expansions 
at a candidate pole $\lambda=
-\lambda_j \in P$ have 
the following form: 
\begin{equation}
\frac{a_{j,k_j}^{\pm}
(\varphi) }{(\lambda +\lambda_j)^{k_j}}+ 
\cdots \cdots + 
\frac{a_{j,2}^{\pm}
(\varphi)}{(\lambda +\lambda_j)^{2}}+ 
\frac{a_{j,1}^{\pm}
(\varphi)}{(\lambda +\lambda_j)}+ 
\cdots \cdots  \qquad (a_{j,k}^{\pm}
(\varphi) \in \RR) 
\end{equation}
(see for example \cite{A-G-Z-V}, \cite{Kan} etc.). 
By the proof of 
Theorem \ref{VT} we obtain a vanishing 
theorem also for the coefficients 
$a_{j,k}^{\pm}(\varphi)$ of 
the poles of $Z_f^{\pm}(\varphi)$. 
\begin{theorem}\label{VT2}
\par \noindent (i) 
Let $1 \leq k \leq k_j$. 
Assume that $\lambda_j$ is not an integer and 
$\{ \alpha \in \ZZ_+^n \ | \ 
c_{\alpha} \not= 0\} \cap 
 \Delta_{j,k} = \emptyset$. 
Then we have $a_{j,k}^{\pm}(\varphi)= \cdots = 
a_{j,k_j}^{\pm}(\varphi)=0$. 
\par \noindent (ii) 
Let $2 \leq k \leq k_j$. 
Assume that $\lambda_j$ is an integer and 
$\{ \alpha \in \ZZ_+^n \ | \ 
c_{\alpha} \not= 0\} \cap 
 \Delta_{j,k-1} = \emptyset$. 
Then we have $a_{j,k}^{\pm}(\varphi)= \cdots = 
a_{j,k_j}^{\pm}(\varphi)=0$. 
\end{theorem}

\section{Explicit formulas for the poles of 
local zeta functions}\label{sec:4}

In this section we give some explicit 
formulas for the coefficients 
$a_{j,n}(\varphi)$, $a_{j,n}^{\pm}(\varphi)$ 
of the deepest poles 
$\lambda =- \lambda_j \in P$ of the 
local zeta functions $Z_f(\varphi)$, 
$Z_f^{\pm}(\varphi)$ introduced in 
Section \ref{sec:3}. We inherit the situation 
and the notations in Section \ref{sec:3}.
Let $-\lambda_j \in P$ be a candidate pole of 
$Z_f(\varphi)$. 
\begin{definition}
For $\sigma \in \Sigma^{(n)}_j$ and $\alpha 
\in \ZZ_+^n$ we define an 
integer $\mu(\sigma, \alpha)_i$ by 
\begin{equation}
\mu(\sigma, \alpha)_i=\nu(\sigma)_i 
- \langle a^i(\sigma), \alpha \rangle 
\in \ZZ. 
\end{equation}
\end{definition} 

\begin{theorem}\label{THA}
Assume that $\lambda_j$ is not 
an odd integer 
and $k_j=n$. Then the coefficient 
$a_{j,n}(\varphi)$ 
of the deepest possible 
pole $\lambda =-\lambda_j \in P$ 
of $Z_f(\varphi)$ is 
given by 
\begin{eqnarray}
 & a_{j,n}(\varphi) & 
\nonumber\\ 
 & = & \sum_{\alpha \in \Delta_{j,n}} 
 \left\{ 
\sum_{\sigma \in \Sigma^{(n)}_j} 
\left( \prod_{i=1}^n 
\frac{1+(-1)^{\nu(\sigma)_i}}{ 
l(a^i(\sigma)) \times \mu(\sigma , \alpha)_i!}
\right) 
\frac{\partial^{\mu(\sigma , \alpha)_1 + 
\cdots + \mu(\sigma , \alpha)_n} }{ 
\partial y_1^{\mu(\sigma , \alpha)_1} 
\cdots \partial y_n^{\mu(\sigma , \alpha)_n} }
 |f_{\sigma}|^{-\lambda_j} (0) \right\} 
\nonumber\\ 
 & \times & 
\frac{\partial_x^{\alpha} \varphi (0)}{\alpha !}, 
\end{eqnarray}
where for $\mu <0$ we set $\frac{\partial^{\mu}
}{\partial y_i^{\mu} } 
( \cdot )=0$. 
\end{theorem} 
\begin{proof}
Since $\lambda_j$ is not an odd integer, by 
the proof of Theorem \ref{MT} we have 
\begin{eqnarray}
  a_{j,n}(\varphi) & = & 
\sum_{\sigma \in \Sigma^{(n)}_j} 
\left( \prod_{i=1}^n 
\frac{1+(-1)^{\nu(\sigma)_i}}{ 
l(a^i(\sigma)) \times \nu(\sigma)_i!}
\right) 
\nonumber\\ 
 & & \times 
\frac{\partial^{\nu(\sigma)_1 + 
\cdots + \nu(\sigma)_n } }{ 
\partial y_1^{\nu(\sigma)_1} 
\cdots \partial y_n^{\nu(\sigma)_n} }
\left\{ |f_{\sigma}|^{-\lambda_j} 
(\varphi \circ 
\pi (\sigma)) \right\}_{y=0}. 
\end{eqnarray}
Now let 
\begin{equation} 
\varphi (x)= \sum_{\alpha \in \ZZ_+^n} 
\frac{\partial_x^{\alpha} \varphi (0)}{\alpha !}
x^{\alpha} 
\end{equation} 
be the Taylor expansion of $\varphi$ 
at $0 \in \RR^n$. Since $a_{j,n}(x^{\alpha})=0$ 
for $\alpha \notin \Delta_{j,n}$ by 
Theorem \ref{VT} 
and we have 
\begin{equation} 
(x^{\alpha} \circ \pi (\sigma))(y) \equiv 
y_1^{\langle a^1(\sigma), \alpha \rangle} 
\cdots 
y_n^{\langle a^n(\sigma), \alpha \rangle}
\end{equation} 
in a neighborhood of $0 \in \RR^n(\sigma)$ 
for any $\sigma \in \Sigma_j^{(n)}$, 
we obtain 
\begin{eqnarray}
 & a_{j,n}(\varphi) & = 
  \sum_{\alpha \in \Delta_{j,n}} 
\nonumber\\ 
 & & 
\left\{ 
 \sum_{\sigma \in \Sigma^{(n)}_j} 
\left( \prod_{i=1}^n 
\frac{1+(-1)^{\nu(\sigma)_i}}{ 
l(a^i(\sigma)) \times \nu(\sigma)_i!}
\right) 
\frac{\partial^{\nu(\sigma)_1 + 
\cdots + \nu(\sigma)_n } }{ 
\partial y_1^{\nu(\sigma)_1} 
\cdots \partial y_n^{\nu(\sigma)_n} }
\left( |f_{\sigma}|^{-\lambda_j} 
y_1^{\langle a^1(\sigma), \alpha \rangle} 
\cdots 
y_n^{\langle a^n(\sigma), \alpha \rangle} 
 \right)_{y=0} 
\right\} 
\nonumber\\ 
 & & \times 
\frac{\partial_x^{\alpha} 
\varphi (0)}{\alpha !}. 
\end{eqnarray}
Then the result follows from the Leibniz rule. 
This completes the proof.
\end{proof}
In order to state similar results for 
$a_{j,n}^{\pm}(\varphi)$ we define two integers 
$c_{\pm}(\sigma)$ ($\sigma \in \Sigma^{(n)}_j$) 
as follows. First set $\{ \pm 1 \}^n :=
\{ \e =(\e_1, \e_2 , \ldots , \e_n) \ | \ 
\e_i=\pm 1\}$. For $\sigma \in \Sigma^{(n)}_j$ 
we define subset $Q_{\pm}(\sigma)$ of 
 $\{ \pm 1 \}^n $ by 
\begin{equation}
Q_{\pm}(\sigma)=\{ \e =(\e_1, \e_2 , 
\ldots , \e_n) \ | \ 
\pm f_{\sigma}(0) \times 
\prod_{i=1}^n \e_i^{l(a^i(\sigma))} >0\}. 
\end{equation}
Let us explain the meaning of $Q_{\pm}(\sigma) 
 \subset \{ \pm 1 \}^n $. For each 
$\e =(\e_1, \ldots , \e_n) \in \{ \pm 1 \}^n $ 
we define an open subset $V_{\e}$ of 
$\RR^n(\sigma) \simeq \RR^n_y$ by 
\begin{equation}
V_{\e}=\{ (y_1, y_2, \ldots , y_n) 
\in \RR^n(\sigma) \ | \ 
\e_iy_i >0 \ \text{for any} \ 
1 \leq i \leq n \}. 
\end{equation}
Then there exists a sufficiently small 
open neighborhood $W$ of 
$0 \in \RR^n(\sigma)$ such that 
$\pm (f \circ \pi (\sigma))|_{W \cap V_{\e}}>0$ 
for any $\e \in Q_{\pm}(\sigma)$. 
Namely $Q_{\pm}(\sigma)$ is naturally identified 
with the set $\{ V_{\e} \}_{\e \in Q_{\pm}(\sigma)}$ 
of open quadrants in $\RR^n(\sigma) \simeq \RR^n_y$ 
such that $\pm (f \circ \pi 
(\sigma))|_{W \cap V_{\e}}>0$ in a neighborhood 
of $0 \in \RR^n(\sigma)$. 
\begin{definition}
For $\sigma \in \Sigma^{(n)}_j$ we set 
\begin{equation}
c_{\pm}(\sigma)=\sum_{\e \in Q_{\pm}(\sigma)} 
\left( \prod_{i=1}^n \e_i^{\nu(\sigma)_i} 
\right) \in \ZZ . 
\end{equation}
\end{definition} 
\noindent Note that for any 
$\sigma \in \Sigma^{(n)}_j$ we 
have 
\begin{equation}
c_{+}(\sigma)+ c_{-}(\sigma)= 
\prod_{i=1}^n 
\{ 1+ (-1)^{\nu(\sigma)_i} \} . 
\end{equation}
\begin{theorem}\label{THB} 
Assume that $\lambda_j$ is not an integer 
and $k_j=n$. Then the coefficient 
$a_{j,n}^{\pm}(\varphi)$ 
of the deepest possible 
pole $\lambda =-\lambda_j \in P$ 
of $Z_f^{\pm}(\varphi)$ is 
given by 
\begin{eqnarray}
 & a_{j,n}^{\pm}(\varphi) & 
\nonumber\\ 
 & = & \sum_{\alpha \in \Delta_{j,n}} 
\left\{ 
\sum_{\sigma \in \Sigma^{(n)}_j}
c_{\pm}(\sigma) 
\left( \prod_{i=1}^n 
\frac{1}{ 
l(a^i(\sigma)) \times \mu(\sigma , \alpha)_i!}
\right) 
\frac{\partial^{\mu(\sigma , \alpha)_1 + 
\cdots + \mu(\sigma , \alpha)_n} }{ 
\partial y_1^{\mu(\sigma , \alpha)_1} 
\cdots \partial y_n^{\mu(\sigma , \alpha)_n} }
 |f_{\sigma}|^{-\lambda_j} (0)
 \right\} 
\nonumber\\ 
 & & \times 
\frac{\partial_x^{\alpha} \varphi (0)}{\alpha !}, 
\end{eqnarray}
where for $\mu <0$ we set $\frac{\partial^{\mu}
}{\partial y_i^{\mu} } 
( \cdot )=0$. 
\end{theorem}

\section{Asymptotic expansions of oscillating 
integrals}\label{sec:5}

In this section, combining our previous arguments 
with those of \cite{A-G-Z-V} and \cite{V}, we 
obtain some results on the asymptotic 
expansions of oscillating 
integrals. As before, let 
 $f$ be a real-valued real analytic function 
defined on an open neighborhood $U$ of $0 \in \RR^n$ 
and $\varphi \in C_0^{\infty}(U)$ a 
real-valued test function defined 
on $U$. We inherit the notations in Section 
\ref{sec:3} and \ref{sec:4}. Then the oscillating 
integral $I_f(\varphi)(t)$ ($t \in \RR$) 
associated to $f$ and $\varphi$ is defined by 
\begin{equation}
I_f(\varphi)(t) =\int_{\RR^n} 
e^{it f(x)} \varphi (x)dx. 
\end{equation}
Here we set $i=\sqrt{-1}$ for short. By the 
fundamental results of Varchenko \cite{V} 
(see also \cite{A-G-Z-V} and \cite{Kan} 
for the detail), 
as $t \longrightarrow + \infty$ the oscillating 
integral $I_f(\varphi)(t)$ has an 
asymptotic expansion of the form 
\begin{equation}\label{E-A} 
I_f(\varphi)(t) \sim \sum_{j=1}^{\infty} 
\sum_{k=1}^{k_j} c_{j,k}(\varphi) 
t^{-\lambda_j}(\log t)^{k-1}, 
\end{equation}
where $c_{j,k}(\varphi)$ are some complex 
numbers. Despite the important contributions 
by many mathematicians (see for example 
\cite{A-G-Z-V}, \cite{D-N-S}, \cite{G-S} 
and \cite{S} etc.), little 
is known about the coefficients $c_{j,k}(\varphi)$ 
of the asymptotic expansion. First of all, 
we shall give a general vanishing theorem 
for these coefficients $c_{j,k}(\varphi)$. 
Let us fix a candidate pole $-\lambda_j 
\in P$ of the local zeta function 
$Z_f(\varphi)$ and let 
\begin{equation}
\varphi (x)=\sum_{\alpha \in \ZZ_+^n} 
c_{\alpha} x^{\alpha} \qquad 
(c_{\alpha} \in \RR) 
\end{equation}
be the Taylor expansion of the test 
function $\varphi$ at $0 \in U \subset \RR^n$. 

\begin{theorem}
\par \noindent (i) 
Let $1 \leq k \leq k_j$. 
Assume that $\lambda_j$ is not an integer and 
$\{ \alpha \in \ZZ_+^n \ | \ 
c_{\alpha} \not= 0\} \cap 
 \Delta_{j,k} = \emptyset$. 
Then we have $c_{j,k}(\varphi)= \cdots = 
c_{j,k_j}(\varphi)=0$. 
\par \noindent (ii) 
Let $2 \leq k \leq k_j$. 
Assume that $\lambda_j$ is an integer and 
$\{ \alpha \in \ZZ_+^n \ | \ 
c_{\alpha} \not= 0\} \cap 
 \Delta_{j,k-1} = \emptyset$. 
Then we have $c_{j,k}(\varphi)= \cdots = 
c_{j,k_j}(\varphi)=0$. 
\end{theorem} 

\begin{proof}
By using the results of \cite{A-G-Z-V} 
and \cite{V}, the result follows 
immediately from Theorem \ref{VT2}. 
\end{proof}
\noindent 
Next we give an explicit formula for 
 the coefficient $c_{j,n}(\varphi)$ 
of $t^{-\lambda_j} (\log t)^{n-1}$ 
 in the asymptotic expansion \eqref{E-A}. 
For this purpose, we define two 
real numbers $b_{j,n}^{\pm}(\varphi) 
\in \RR$ by 
\begin{eqnarray}
 & b_{j,n}^{\pm}(\varphi) & 
\nonumber\\ 
 & = & \sum_{\alpha \in \Delta_{j,n}} 
\left\{ 
\sum_{\sigma \in \Sigma^{(n)}_j}
c_{\pm}(\sigma) 
\left( \prod_{i=1}^n 
\frac{1}{ 
l(a^i(\sigma)) \mu(\sigma , \alpha)_i!}
\right) 
\frac{\partial^{\mu(\sigma , \alpha)_1 + 
\cdots + \mu(\sigma , \alpha)_n} }{ 
\partial y_1^{\mu(\sigma , \alpha)_1} 
\cdots \partial y_n^{\mu(\sigma , \alpha)_n} }
 |f_{\sigma}|^{-\lambda_j} (0) 
\right\} 
\nonumber\\ 
 & & \times 
\frac{\partial_x^{\alpha}
 \varphi (0)}{\alpha !}, 
\end{eqnarray}
where for $\mu <0$ we set $\frac{\partial^{\mu}
}{\partial y_i^{\mu} } 
( \cdot )=0$. Recall that if $\lambda_j$ is 
not an integer we have $a_{j,n}^{\pm}(\varphi)
=b_{j,n}^{\pm}(\varphi)$.  
\begin{theorem}
 The coefficient $c_{j,n}(\varphi)$ 
of $t^{-\lambda_j} (\log t)^{n-1}$ 
 in the asymptotic expansion 
\eqref{E-A} of $I_f(\varphi)$ 
is given by 
\begin{equation}
c_{j,n}(\varphi)=
\frac{\Gamma (\lambda_j)}{(n-1)!} 
\left(
 e^{\frac{\pi i}{2}\lambda_j}b^+_{j,n}(\varphi) 
+ e^{-\frac{\pi i}{2}\lambda_j}b^-_{j,n}(\varphi) 
\right) . 
\end{equation}
\end{theorem} 
\begin{proof}
We use the notations in the proof 
of Theorem \ref{MT}. 
By the arguments in \cite{A-G-Z-V} 
and \cite{Kan} etc. 
we do not have to consider the 
contributions from $\{ f_{\sigma_q}=0\}$ 
($q=1,2, \ldots ,N$). 
Then the result follows from 
(the proof of) Theorem \ref{THA} 
and \ref{THB}.  
\end{proof}

\end{document}